\documentclass[11pt,oneside]{amsart}

\usepackage{amsmath,ifthen, amsfonts, amssymb,
srcltx, amsopn, color, enumerate, mathrsfs, overpic, hyperref}
\usepackage[cmtip,arrow]{xy}

\usepackage{pb-diagram, pb-xy}
\usepackage{marginnote}
\usepackage{overpic}
\usepackage{subfig}

\usepackage[T1]{fontenc}
\usepackage{yfonts}

\usepackage{graphicx}

\newcommand{\showcomments}{yes}
\renewcommand{\showcomments}{no}

\newsavebox{\commentbox}
{\ifthenelse{\equal{\showcomments}{yes}}%
{\footnotemark
        \begin{lrbox}{\commentbox}
        \begin{minipage}[t]{1.25in}\raggedright\sffamily\tiny
        \footnotemark[\arabic{footnote}]}
{\begin{lrbox}{\commentbox}}}%
{\ifthenelse{\equal{\showcomments}{yes}}%
{\end{minipage}\end{lrbox}\marginpar{\usebox{\commentbox}}}
{\end{lrbox}}}

\newtheorem{thm}{Theorem}

\newtheorem{lem}[thm]{Lemma}

\newtheorem{cor}[thm]{Corollary}

\newtheorem{prop}[thm]{Proposition}

\theoremstyle{definition}
\newtheorem{defn}[thm]{Definition}

\newtheorem*{rk}{Remark}

\newtheorem{exmp}[thm]{Examples}

\newtheorem{claim}[thm]{Claim}

\newtheorem{claim*}{Claim}

\DeclareMathOperator{\Aut}{Aut}

\DeclareMathOperator{\stabilizer}{Stab}

\newcommand{\Z}{\ensuremath{\mathbb{Z}}}

\newcommand{\R}{\ensuremath{\mathbb{R}}}

\makeatletter

\newcommand{\Rmnum}[1]{\mathbf{{\expandafter\@slowromancap\romannumeral #1@}}}

\makeatother

\let\oldmarginpar\marginpar
\renewcommand\marginpar[1]{\-\oldmarginpar[\raggedleft\footnotesize #1]%
{\raggedright\footnotesize #1}}

\newcounter{enumitemp}

\newcommand{\N}{\ensuremath{\mathbb{N}}}

\begin{document}
\title{Groups with near exponential residual finiteness growth}
\author[K. Bou-Rabee]{Khalid bou-Rabee}
\thanks{K.B. supported in part by NSF grant DMS-1405609, A.M. supported by Swiss NSF grant 200021\_144323 and P2GEP2\_162064.}
\address{School of Mathematics, CCNY CUNY, New York City, New York, USA}
\email{khalid.math@gmail.com}
\author[A. Myropolska]{Aglaia Myropolska}
\address{Laboratoire de Math\'ematiques, 
Universit\'e Paris-Sud 11, Orsay, France}
\email{aglaia.myropolska@math.u-psud.fr}

\date{\today}

\begin{abstract}
A function $\N \to \N$ is \emph{near exponential} if it is bounded above and below by functions of the form $2^{n^c}$ for some $c > 0$.
In this article we develop tools to recognize the near exponential residual finiteness growth  in groups acting on rooted trees. In particular, we show the near exponential residual finiteness growth for certain branch groups, including the first Grigorchuk group, the family of Gupta-Sidki groups and their variations, and Fabrykowski-Gupta groups. We also show that the family of Gupta-Sidki p-groups, for $p\geq 5$, have super-exponential residual finiteness growths.
\end{abstract}

\maketitle
\setcounter{tocdepth}{2}

\tableofcontents

\section{Introduction}\label{sec:introduction}

The notion of residual finiteness growth (depth function) measures how efficiently finite groups approximate a given group. In this article, we begin a stratification of a well-known class of non-linear groups via residual finiteness growths. 
This class consists of groups admitting a ``nice'' action on a $d$-regular rooted tree\footnote{That is, a tree with the distinguished vertex $\emptyset$ of degree $d$ and all other vertices of degree $d+1$.}. This is the class of \emph{branch groups}: groups admitting a lattice of subnormal subgroups with the branching structure following the structure of the tree on which the group acts. 
The class of branch groups, defined in \cite{MR0274575} and \cite{MR1765119}, is one of three classes that 
partition the class of all just-infinite groups, that is infinite groups all of whose proper quotients are finite. Furthermore, the class of branch groups contains many examples of groups with remarkable algebraic properties.
One of them is the \emph{first Grigorchuk group} \cite{MR565099}, $\Gamma$, that comes equipped with a natural embedding into the automorphism group of a rooted binary tree, $T_2$. This group is far from being linear: it is a just-infinite 2-group, it is commensurable with $\Gamma \times \Gamma$, and has intermediate word growth (see \cite{MR2195454} for a survey on $\Gamma$). 
Moreover, the group $\Gamma$ has exponential depth function \cite{B09}. 
Before we state our results, we recall the definition of some residual finiteness growth functions.

Let $G$ be a finitely generated residually finite group. 
The \emph{depth function of an element} $g\in G \setminus \{1 \}$ is defined as follows
$$
D_G(g) = \min \{ |G: N|, \: N\lhd G \text{ and } g\notin N \}.
$$
For a fixed finite generating set $S$ of $G$ and $g\in G$ denote by $||g||_S$ the word length of $g$ with respect to $S$. 
Define the \emph{residual finiteness growth} as
$$
F_G^S(n)=\max_{g\in G\setminus \{1\} :\; ||g||_S\leq n} D_G(g).
$$

Let $B_G^S(n)=\{g\in G\mid ||g||_S\leq n\}$ be the word metric $n$-ball. Define the \emph{full residual finiteness growth} $\Phi_G^S$ as
$$
\Phi_G^S(n)=\min \{ |Q| : B_G^S(n) \text{ injects into } Q \text{ through an epimorphism } \phi: G \to Q  \}.
$$
Clearly, one has $F_G^S(n)\leq \Phi_G^S(n)$. 

For two functions $f, g \colon \R \rightarrow \N$ we write $f \preceq g$ if there exists $C > 0$ such that $f(n) \leq g(Cn)$. 
We say that $f$ and $g$ are equivalent ($f \approx g$) if $f \preceq g$ and $g \preceq f$. 
It follows from \cite[Lemma 1.1]{B09} and \cite[Lemma 1.1]{BS14} that for two finite generating sets $S_1$ and $S_2$ of $G$, the following equivalences hold: $F_G^{S_1}\approx  F_G^{S_2}$ and $\Phi_G^{S_1}\approx  \Phi_G^{S_2}$. 
We will denote the equivalence class of the depth functions $F_G^S$ and full depth function $\Phi_G^S$ of the group $G$ by $F_G$ and $\Phi_G$ respectively.
%

We are now ready to state our results. 
In our first result, we will already see that the concept of depth function is appropriate when dealing with branch groups, as it quantifies how far down the tree the group  acts nontrivially. The reader can find the precise definitions of regular branch groups, contracting property and congruence subgroup property  in \S \ref{sec:prelim}.

\begin{thm} \label{MainResultUpper}
Let $H$ be a finitely generated group acting on a rooted $d$-regular tree. Suppose that $H$ is regular branch and contracting with contraction coefficient $\lambda<1$.
Then 
$$\Phi_H(n) \preceq 2^{n^{\frac{1}{\log_d(1/\lambda)}}}.$$
\end{thm}

\begin{rk}
In \cite[Lemma $2.13.9$]{MR2162164} it is shown that the contracting coefficient $\lambda$ satisfies 
 $\frac{1}{\log_d(1/\lambda)}\geq 1$. 
 Thus, any upper bound achieved by Theorem \ref{MainResultUpper} is at least exponential.
\end{rk}

\smallskip
For just-infinite regular branch groups with the congruence subgroup property, we can find that their growths are super-polynomial.

\begin{thm} \label{MainResultLower}
Let $H$ be a finitely generated just-infinite group acting on a rooted $d$-regular tree. Suppose that $H$ is regular branch with the congruence subgroup property. 

Then $F_H(n) \succeq 2^{n^{\frac{1}{\log_d(\delta)}}}$ for some $\delta=\delta(H)>1$.
\end{thm}

A map is \emph{near exponential} if it is bounded above and below by expressions of the form $2^{n^c}$ for $c > 0$.
We immediately obtain the following.

\begin{cor} \label{cor:main}
	Let $H$ be a finitely generated just-infinite regular branch contracting group with the congruence subgroup property. Then $\Phi_H$ and $F_H$ are both near exponential.
\end{cor}

\noindent
It is remarkable that Corollary \ref{cor:main} indicates that regular branch contracting groups with the congruence subgroup property sit in the class of nonlinear groups in a way analogous to how arithmetic groups, and even nilpotent groups, sit in the class of linear groups. That is, the residual finiteness growths of these groups do not wander that far away from each other.
It was shown in \cite{MR2925403} that arithmetic groups have precisely polynomial growth of a fixed degree. 
In \cite{B09}, it was shown that nilpotent groups have polynomial in logarithm residual finiteness growth (and in \cite{BS14} it is shown that the full residual finiteness growth of many nilpotent groups is precisely $n^b$ for some positive integer $b$).
Finally, it was shown that all finitely generated linear groups have polynomial residual finiteness growth in \cite{BM13}.
So while the class of arithmetic groups have residual finiteness growth clustering around polynomials, and nilpotent groups around polynomial in log functions, the class of branch groups clusters around exponential functions.

We can strengthen the conclusion of Corollary \ref{cor:main} in some cases, which includes the first Grigorchuk group. We prove this in \S \ref{sec:MainResultUpperLower}.
\begin{cor} \label{MainResultUpperLower}
	Let $H$ be a finitely generated just-infinite group acting on a $d$-regular tree. 
	Assume that 		    
	$H$ is 
	\begin{enumerate}
	\item regular branch with the congruence subgroup property;
	\item contracting with the contraction coefficient $\lambda<1$.
\end{enumerate}		
	
	Suppose there exists a sequence of nontrivial elements $h_i \in H$ such that 
	$h_i \in \stabilizer_H(i)$ and $\|h_i\| \leq \lambda^{-i}$.
	Then
	$$
		\Phi_H(n) \simeq 2^{n^{\frac{1}{\log_d(1/\lambda)}}}.
	$$
\end{cor}

The class of regular branch groups is rich and well-studied. Please see Examples \ref{ManyExamples} for a quick overview of existing examples.

Our next result shows that the conclusion of Corollary \ref{cor:main} cannot be strengthened over the class of all finitely generated just-infinite regular branch contracting groups with the congruence subgroup property. We prove this result in \S \ref{sec:perv}.

\begin{thm} \label{prop:gupta}
Let $r > 0$. Then there exists a prime $p$ such that if $G_p$ is the Gupta-Sidki $p$-group, then 
$$F_{G_p} (n) \succeq 2^{n^r}.$$
Moreover, $G_p$ for $p\geq 5$ have super-exponential residual finiteness growths.
\end{thm}

We have also developed methods to deduce near exponential full residual finiteness growth under some weaker assumptions on a group. We prove the following in \S \ref{sec:perv}. Note that the Pervova group does not have the congruence subgroup property \cite{MR2308183}.


\begin{prop} \label{prop:perv}
	The Pervova group has near exponential $\Phi_G$ growth.
\end{prop}
\noindent

As an application of our results, the polynomial residual finiteness growth can be used to distinguish non-linear groups. Namely, showing that a group does not have polynomial depth function is a way to show it is not linear \cite[Theorem 1.1]{BM13}. Hence applying Theorem \ref{MainResultLower} we have the following.

\begin{cor} \label{MainApplication}
A finitely generated just-infinite regular branch group with the congruence subgroup property is not linear.
\end{cor}

A more general result on non-linearity of any weakly branch (and, therefore, any branch) group was shown in \cite{MR2273978}.

It would be interesting to determine whether intermediate growth occurs in the class of branch groups (or even in the class of all finitely generated groups).

\smallskip
This article is organized as follows.
In \S \ref{sec:rfgrowth} we present notation on residual finiteness growth functions.
In \S \ref{subsec:branch} we present notation and prove some basic properties about branch groups, self-similar and contracting groups.
In \S \ref{sec:proofs} we give proofs of Theorems \ref{MainResultUpper} and \ref{MainResultLower}. 
In \S \ref{sec:perv} we give proofs of Theorem \ref{prop:gupta} and Proposition \ref{prop:perv}.

\subsection*{Acknowledgements}
We thank Pierre de la Harpe for a number of corrections and suggestions.


\section{Preliminaries} \label{sec:prelim}

\subsection{Residual finiteness growth} \label{sec:rfgrowth}
Let $G$ be a finitely generated residually finite group with a finite generating set $S$ and let $\phi: G\to Q$ be an epimorphism onto a finite group $Q$. 
We say that a set $A \subseteq G$ is \emph{detected} by $Q$ if $A\cap \ker \phi \subset \{1\}$. 
We say $A$ is \emph{fully detected} by $Q$ if $\phi |_A$ is an injection. 
Using this notation, $D_G(g)=\min \{|Q| : Q \text{ detects } \{g\}\}$. 
Then $F_G^S(n)$ is defined to be the maximal value of $D_G(g)$ over $B_G^S(n)$, the word metric ball of radius $n$ with respect to $S$. 
Further, $\Phi_G(n)$ is defined to be the minimal finite quotient $Q$ of $G$ that fully detects $B_G^S(n)$. 

We define $F_G$ and $\Phi_G$ as equivalence classes of functions whose values do not depend on generating set. 
When values are explicitly computed for an $n \in \N$, we list the depending on generating set $S$ of $G$ by writing $F_G^S(n)$ and $\Phi_{G}^S(n)$. 

We list some basic properties of $F_G$ and $\Phi_G$ for the convenience of the reader:
\begin{enumerate}
\item $F_\Z(n) \approx \log(n)$ \cite[Theorem 2.2]{B09} while it is easy to see that $\Phi_\Z(n) \approx n$.
\item Let $G$ and  $H \leq G$ be two finitely generated residually finite groups. Then $F_H \preceq F_G$ and $\Phi_H \preceq \Phi_G$.
\item Let $G$ be a group and $H$ a finite-index subgroup of $G$.
Then $F_G \preceq (F_H)^{[G:H]}$ \cite[Lemma 1.2]{B09}.
\item For every finitely generated group $G$, we have $F_G \preceq \Phi_G$.
\item
If $\phi : G \to H$ is a homomorphism between finitely genereated residually finite groups, then there is, in general, no relationship between $F_G$ and $F_H$.
For instance, while any nonabelian free group has growth between $n^{2/3}$ and $n^3$ \cite{Th15}, there exists groups with arbitrarily large residual finiteness growths \cite{BS13b} and the free group maps onto $\mathbb{Z}$.
In some particular cases, one can draw a relationship: see, for instance, \cite[Lemma 2.4]{BK12}.
\end{enumerate}

\subsection{Regular branch, self-similar and contracting groups}
\label{subsec:branch}

The groups we shall consider will all be subgroups of the group $\Aut T$ of automorphisms of a regular rooted tree $T$. Let $X$ be a finite alphabet with $|X|\geq 2$. The vertex set of the tree $T_X$ is the set of finite sequences over $X$; two sequences are connected by an edge when one can be obtained from the other by right-adjunction of a letter in $X$. The root is the empty sequence $\emptyset$, and the children of $v$ are all $v x$ for $x\in X$. The set $X^n\subset T_X$ is called the \emph{$n$th level} of the tree $T_X$. An automorphism of the tree is a bijective morphism of $T_X$.

Let $g \in \Aut T_X$ be an automorphism of the rooted tree $T_X$. 
Consider a vertex $v\in T_X$ and the subtrees $vT_X=\{vw\mid w\in T_X\}$ and $g(v)T_X=\{g(v)w\mid w\in T_X\}$. 
Notice that a map $vT_X\to g(v)T_X$ is a morphism of rooted trees. 
Moreover, the subtrees $vT_X$ and $g(v)T_X$ are naturally isomorphic to $T_X$. Identifying $vT_X$ and $g(v)T_X$ with $T_X$ we get an automorphism $g|_v\colon T_X\to T_X$ uniquely defined by the condition 
$$
g(vw)=g(v)g|_v(w)
$$ for all $w\in T_X$.
We call the automorphism $g|_v$ the \emph{restriction} of $g$ on $v$. Notice the following obvious properties of the restrictions:
\begin{align*}
g|_{v_1 v_2}&=g|_{v_1}|_{v_2}\\
(g_1\cdot g_2)|_v&=g_1|_{g_2(v)}\cdot g_2|_v.
\end{align*}

It follows that the action of the automorphism $g \in \Aut T_X$  can be seen as $\pi_g(g_1,\dots, g_{|X|})$, and we will often write  $$g=\pi_g(g_1,\dots, g_{|X|}),$$ where the permutation $\pi_g \in Sym(X)$ is defined by the action of $g$ on the first level of the tree, and $g_1,\dots, g_d \in \Aut T_X$ are the restrictions of $g$ on the vertices of the first level of $T_X$. 

\medskip
A subgroup $G$ of $\Aut(T_X)$ is \emph{self-similar} if for every $g\in G$ and every $v\in T_X$ the restriction $g|_v \in G$.
\medskip

An obvious example of a self-similar group is $\Aut T_X$ itself.

We further establish a notion of contraction of the self-similar action.

\begin{defn} Let $G\leq \Aut T_X$ be a self-similar finitely generated group with a finite generating set $S$. The number 
\begin{equation}
\label{contraction}
\lambda_{(G,T_X)}=\limsup_{n\rightarrow \infty}\sqrt[n]{\limsup_{\|g\|_S\rightarrow \infty} \max_{v\in X^n}\frac{\|g|_v\|_S}{\|g\|_S}}
\end{equation}
is called the \emph{contraction coefficient} of the action $(G, T_X)$.
Note that the limit in the definition does not depend on the choice of generating set (see \cite[Lemma $2.11.10$]{MR2162164}). 

A self-similar group $G \leq \Aut(T_X)$ is called \emph{contracting} if $\lambda_{(G,T_X)}<1$. In other words, a self-similar group $G \leq \Aut(T_X)$ is 
\emph{contracting} if there exist positive constants $\lambda<1$, $k_0$ and $C$ such that for every element $g\in G$ and every vertex $v\in T_X$ of level $k\geq k_0$ the following inequality holds
\begin{equation*}
\|g|_v\|_S<\lambda^{k} \|g\|_S+C.
\end{equation*}
\end{defn}

We refer the reader to Examples \ref{ManyExamples} and to \cite{MR2162164} for examples of contracting actions. Whilst many self-similar actions were proved to be contracting, finding a method to compute the exact value of the contraction coefficient is an interesting open question.


\medskip
We will need more notation to define branch and regular branch groups. 
Let $G\leq \Aut T_X$ be an automorphism group of the rooted tree $T_X$. 
For a vertex $v\in T_X$ the \emph{vertex stabilizer} is the subgroup consisting of the automorphisms that fix the sequence $v$: $$\stabilizer_G(v)=\{g\in G\mid g(v)=v\}.$$ 
The \emph{$n$-th level stabilizer} (also called a \emph{principal congruence subgroup}) is the subgroup $\stabilizer_G(n)$ consisting of the automorphisms that fix all vertices of the $n$th level:  $$\stabilizer_G(n)=\cap_{v\in X^n} \stabilizer_G(v).$$
Stabilizer subgroups $\stabilizer_G(n)$ with $n\geq 0$ are normal in $G$. 

Notice that any $g\in \stabilizer_G(n)$ can be identified in a natural way with the collection $g_1, \dots, g_{|X|^n}$ of elements of $\Aut T_X$ where $g_i=g|_{v}$ is the restriction of $g$ on the vertex $v$ of level $n$ having the number $i$ in the natural ordering of the vertices in the $n$-the level $(1\leq i\leq |X|^n)$. We say that $g$ is \emph{of level $n$} if $g\in \stabilizer_G(n)\setminus \stabilizer_G(n+1)$ and we write $g=(g_1, \dots, g_{|X|^n})_n$. 

The \emph{rigid stabilizer} $\operatorname{rist}_G(v)$ of a vertex $v\in T_X$ is the subgroup of $G$ of all automorphisms acting non-trivially only on the vertices of the form $vu$ with $u\in T_X$: $$\operatorname{rist}_G(v)=\{g\in G\mid g(w)=w \text{ for all } w\notin vT_X\}$$
The \emph{$n$th level rigid stabilizer} $$\operatorname{rist}_G(n)=\langle \operatorname{rist}_G(v)\mid v\in X^n\rangle$$ is the subgroup generated by the union of the rigid stabilizers of the vertices of the $n$th level.


We say that a subgroup $K$ \emph{geometrically contains} $K^{|X|^i}$, for some $i\geq 1$, if for every $k_1, \dots, k_{|X|^i} \in K$ there exists an element $k\in K$ such that $k=(k_1, \dots, k_{|X|^i})_i$.

\medskip

We say that a level transitive group\footnote{\emph{i.e.} $G$ acts transitively on each level of $T_X$.} $G\leq Aut T_X$ is \emph{branch} if $\operatorname{rist}_G(n)$ is of finite index in $G$ for all $n \geq 1$. In this article we will restrict ourselves to the particularly important type of branch groups introduced by the following definition. 

\begin{defn}
A level transitive group $G\leq \Aut T_X$ is \emph{regular branch} if there exists a finite index subgroup $K$ of $G$ such that $K$ geometrically contains $K^{|X|}$ of finite index. 
\end{defn}

\begin{lem} \label{lem:congsizes}
Let $X$ be a finite set with $|X|\geq 2$ and let $T_X$ be the rooted regular tree as above. Let $G$ be a regular branch group acting on $T_X$.
Then there exist $C, D > 0$ such that
$$2^{|X|^{i}} \leq |G / \stabilizer_G(i)| \leq C^{D\cdot |X|^{\cdot i}}.$$
\end{lem}

\begin{proof}
Let $d=|X|\geq 2$ and let $K$ be a finite index subgroup of $G$ such that $K$ geometrically contains $K^{d}$ of finite index. 

The upper bound follows from the following. Let $m:=
\max\{|G:K|, |K:K^{d}|\}<\infty$. 
Notice that $K$ also geometrically contains the subgroup $K^{d^i}$ for $i\geq 2$. Observe that $|K:~K^{d^i}|\leq m^{d^i}$; indeed, it can be shown by induction that $|K:~K^{d^i}|\leq |K: K^{d^{i-1}}|\cdot |K: K^{d}|^{d^{i-1}}\leq m^{d^{i}}$. It follows that the subgroup $K^{d^i}$ is of index at most $m\times m^{d^i}$ in $G$. Since $K^{d^i}\leq \stabilizer_G(i)$ then $|G/\stabilizer_G(i)|\leq m^{d^i+1}$.

The lower bound follows from the following. Suppose that all elements $k\in K$ are of level at least $i\geq 0$, \emph{i.e.} $K=\stabilizer_G(i)\cap K$ and there exists $k\in K$ which acts nontrivially on the $(i+1)$th level. 
Then $k$ does not belong to the trivial coset of $\stabilizer_G(i+1)$ and thus the index of $\stabilizer_G(i+1)$ in $G$ is at least $2$. 
Observe that $K^d$ acts nontrivially on the $(i+2)$-nd level such that the elements 
$$(k,1,1,\dots, 1)_1, (1,k,1,\dots, 1)_1, \dots, (1, 1, \ldots, 1, k)_1, $$
define different non-trivial cosets of $\stabilizer_G(i+2)$. 
It follows that the index of  $\stabilizer_G(i+2)$ is bounded below by $2^{d}$. Inductively $K^{d^{r}}$ acts nontrivially on the $(i+r+1)$th level and moreover there are at least $2^{d^{r}}$ elements which define different cosets of $\stabilizer_G(i+r+1)$; therefore the index of $\stabilizer_G(i+r+1)$ is at least  $2^{d^{r}}$.
\end{proof}

A subgroup $G$ of $\Aut T_X$ is said to satisfy the \emph{congruence subgroup property} if any finite index subgroup $H$ of $G$ contains a principal congruence subgroup $\stabilizer_G(n)$ for some $n\geq 1$. 

A subgroup $G$ of $\Aut T_X$ is said to satisfy the \emph{quantitative congruence subgroup property} if there exists $N \in \N$ such that any normal subgroup $\Delta \leq G$ of finite index in $G$ containing an element of level $n$ contains $\stabilizer_G(n+N)$. It follows from \cite[Theorem $4$]{MR1765119} that a regular branch just-infinite group $G$ with the congruence subgroup property satisfies the quantative congruence subgroup property (see \cite[Proposition $3.9$]{MR1899368} for the details). The quantitative congruence subgroup property is a useful tool for estimating the depth function of a group.

\medskip
We further give examples of self-similar contracting regular branch groups with congruence subgroup property.
\begin{exmp} \label{ManyExamples}
\leavevmode
\begin{enumerate}
\item 
Let $X=\{1,2\}$. 
We will be interested in the automorphisms of $T_X$ defined inductively by:
$$
a = \sigma,\; b = (a,c),\; c = (a,d), \text{ and } d = (1,b),
$$
where $\sigma$ is the transposition $(1,2)\in Sym(X)$.

Let the \emph{first Grigorchuk group} be $\Gamma:=\left< a, b, c, d \right>$. Clearly, $\Gamma$ is self-similar. Moreover, it is $\frac{1}{2}$-contracting, just-infinite and regular branch \cite{MR764305} over the subgroup $K=\langle (ab)^2, (bada)^2, (abad)^2\rangle$. It also has the congruence subgroup property \cite[Proposition 10]{MR1765119}.

\item Let $X=\{1,\dots, p\}$ where $p$ is odd prime.
We will be interested in the automorphisms $x$ and $y$ of $T_X$ defined inductively:
$$x= \sigma, \text{ } y =(x,x^{-1}, 1,\dots, 1,y ),$$
where $\sigma$ is the cyclic permutation $(1,2,\dots p)$ on $X$. 
Let the Gupta-Sidki $p$-group be  $G_p=\langle x,y\rangle$. Clearly, $G_p$ is self-similar. It is moreover contracting, just-infinite and regular-branch over its commutator subgroup \cite{MR759409, MR767112}. Moreover, $G_p$ has the congruence subgroup property (see \cite[Proposition 2.6]{Garrido2014}).

\item There are various modifications of the Gupta-Sidki group which are self-similar, just-infinite, regular branch contracting groups having the congruence subgroup property. Here is an example of such a modification. Let $G$ be the subgroup of automorphisms on the rooted $p$-regular tree for $p\geq 7$ generated by $x=(1,2,\dots, p)$ and $y=(x^{i_1}, x^{i_2}, \dots, x^{i_p-3},1,1,1,y)$ for $0\leq i_j\leq p-1$ and $i_1\neq 0$. The group $G$ is regular branch over its commutator subgroup (see \cite[Example $10.2$]{MR1765119});

\item The Fabrykowsky-Gupta group $\mathcal{G}$ acting by automorphisms on a rooted ternary tree and generated by $a=(1,2,3)$ and $b=(a,1,b)$ is contracting, regular branch, just-infinite, virtually torsion free group with the congruence subgroup property (see \cite[6.2, 6.4]{MR1899368}). 
A natural generalization of the Fabrykowsky-Gupta example is a group $\mathcal{G}_p$ generated by automorphisms $a=(1,2,\dots, p)$ and $b=(a,1,\dots, 1, b)$ of a $p$-regular tree. For every prime $p\geq 5$, the Fabrykowsky-Gupta group $\mathcal{G}_p$ is regular branch just-infinite with the congruence subgroup property due to \cite[Example 10.1]{MR1765119} and, moreover, contracting (to see this use the equivalent definition of contracting action in  \cite{MR2162164}).
\end{enumerate}
\end{exmp}


\section{General branch group bounds} \label{sec:proofs}

\subsection{The proof of Theorem \ref{MainResultUpper}} \label{subsec:proof1}

Let $H$ be a finitely generated regular branch contracting group (acting on a rooted $d$-regular tree) and fix a generating set of $H$. 
For every $g \in H$, by the contracting property of $H$ there exist positive constants $\lambda<1$, $k_0$ and $C$ such that 
\begin{equation}
\| g|_v \| <\lambda^k\|g\|+C\leq \lambda^{k-k_0} \| g \|  +C.
\label{ineq}
\end{equation}
for $v\in T_d$ of level $k\geq k_0$.

Suppose $\|g\|=n$ for some $n\geq 1$. Consider the action of $g$ on the $k$-th level with $k=k_0-\log_{\lambda}(n)\geq k_0$. By (\ref{ineq}) we have $\| g|_v \| < \lambda^{-\log_\lambda(n)}\cdot n+C\leq 1+C$. Thus at level $k_0-\log_\lambda(n)$, we have that there exists at most $|B_H(1+C)|$ choices for the projection. It follows that the level of $g$ is at most $-\log_\lambda (n) + k_0 + D$ where $D$ is the greatest level for each of the finitely many choices of the projections for $g$ at level $k$.
Thus we may detect nontrivial $g\in B_H(n)$ by $H / \stabilizer_H(-\log_\lambda (n)+k_0+D)$. 

Since any element in $B_H(n)$ may be detected by $H/ \stabilizer_H(-\log_\lambda(n)+k_0+D)$ it follows that $B_H(n/2)$ injects into $H/\stabilizer_H(-\log_\lambda(n)+k_0+D)$. By Lemma \ref{lem:congsizes} we have $C, D > 1$ such that for all $i\geq 0$,
$$|H: \stabilizer_H(i)|\leq C^{D d^i},$$ 
which gives us the upper bound, 
$$C^{D d^{-\log_\lambda (n)+k_0+D}}$$ for $\Phi_H(n)$.
Note that $-\log_\lambda(n) = -\log_d(n)/\log_d(\lambda) = \log_d(n)/\log_d(1/\lambda)$, so the upper bound is equivalent to
$$2^{d^{\log_d(n)/(\log_d(1/\lambda))}} = 2^{n^{\frac{1}{\log_d(1/\lambda)}}},$$
thus finishing the proof of Theorem \ref{MainResultUpper}.






\subsection{The proof of Theorems \ref{MainResultLower}}\label{subsec:proof2}

Let $H$ be a finitely generated just-infinite regular branch group (acting on the $d$-regular rooted tree) over $K$ with the congruence subgroup property and let $S$ be its finite generating set.
We first construct candidates that maximize $D_H$ over the word metric balls of radius $n$.

Since $K$ is of finite index in $H$, we have that $K$ is finitely generated.
Fix a generating set $k_1, \ldots, k_m$ for $K$. 
Let $\delta =\max\{\|k_i\|_H, \|(1,\dots, 1, k_i)_1\|_K\}$. Observe that $\delta>1$: indeed, if $\delta=1$ then $K$ is trivial. 

Let $g_1=k_1$ and for $i\geq 2$ let $g_i=(1, \dots, 1, k_1)_{i-1}$.

\begin{claim} \label{claim:wordlengthgi}
For $i\geq 1$ we have $\|g_i\|_K\leq \delta^{i-1}$, and consequently $\|g_i\|_H\leq \delta^i$.
\end{claim}

We prove the claim by induction. For $i=1$ we have $\|g_1\|_K=1$ and suppose  that $\|g_{i-1}\|_K\leq \delta^{i-2}$. 

Consider $g_i=(1,\dots, 1, k_1)_{i-1}$ which is an element of $1\times \dots \times 1\times K$. 
Then $\| g_i\|_{K}\leq \| (1,\dots, 1, k_1)_{i-1}\|_{1\times \dots \times 1\times K}\times D$ where $D$ is the maximal length of generators of $1\times \dots \times 1\times K$ in $K$. Notice, that $D\leq \delta$ and $\| (1,\dots, 1, k_1)_{i-1}\|_{1\times \dots \times 1\times K}=\|(1,\dots, 1, k_1)_{i-2}\|_K=\|g_{i-1}\|_K$. Using the step of induction we have $\|g_i\|_{K}\leq \delta^{i-2}\times \delta=\delta^{i-1}$. 

\bigskip


Let $\Delta$ be a finite index subgroup that does not contain $g_i$.
Since $H$ satisfies the \emph{quantitative congruence subgroup property} \cite[Proposition $3.9$]{MR1899368}, there exists $N \in \N$ such that any finite index subgroup $\Delta \leq H$ containing an element of level $k$ contains $\stabilizer_H(k+N)$.
Thus, if $\Delta$ has an element of level $k$, we have that  $\Delta$ contains $\stabilizer_H(k + N)$.
Since $\Delta$ does not contain $g_i$, an element of level $i-1$, we have that $\Delta$ \emph{cannot} contain
$\stabilizer_H(i-1)$. Thus $k + N > i-1$, giving $k > i - N-1.$
Thus all elements in $\Delta$ are of level at least $i - N-1$
and so by Lemma \ref{lem:congsizes},
$$
[H : \Delta]\geq [H : \stabilizer_H(i - N-1)] \succeq 2^{d^i}.
$$
Thus, by Claim \ref{claim:wordlengthgi}, it follows that $F_H(\delta^i) \succeq 2^{d^i} \implies F_H(n) \succeq 2^{n^{\frac{1}{\log_d(\delta)}}},$ as desired.
This finishes the proof of Theorems \ref{MainResultLower}.

\section{Proof of Corollary \ref{MainResultUpperLower}}
\label{sec:MainResultUpperLower}

In light of Theorem \ref{MainResultUpper}, we need only prove the lower bound.
Let $h_i$ be a sequence in $H$ with $h_i \in \stabilizer_H(i)$ and $\| h_i \| \leq  \lambda^{-i}$.
Following the same arguments as in the proof of Theorem \ref{MainResultLower}, we have
$F_H(\lambda^{-i}) \geq 2^{d^i}$.
Thus,
$$
F_H(n) \succeq 2^{n^{\frac{1}{\log_d(1/\lambda)}}},
$$
as desired.

\section{The Gupta-Sidki and Pervova groups} \label{sec:perv}

In this section, we give explicit computations for the Gupta-Sidki $p$-groups and the Pervova group.
We begin by proving Theorem \ref{prop:gupta}, which gives bounds on the residual finiteness growths of the Gupta-Sidki $p$-groups, $G_p$. Loosely speaking, the idea behind this proof is to find deep elements that are canonically placed in $G_p$ with word lengths that do not depend on $p$.

\begin{proof}[Proof of Theorem \ref{prop:gupta}]
Let $G_p$ be the Gupta-Sidki $p$-group as defined in \S \ref{subsec:branch} and fix generators $x, y, y^x$. Let $\| \cdot \|$ be the metric norm with respect to this generating set.
Let $p \geq 5$.
Set $y_i = x^{i} y x^{-i}$.
Set 
$$
c := [y_{p-1}, y_0]=y_{p-1}^{-1} y_0^{-1} y_{p-1} y_0.
$$
For an element $g \in [G_p,G_p]$, we define 
$$
\| g \|^* := \min \left\{ \sum_{i} (2 \|a_i \| + 4): g = \prod_i a_i [y,x]^{\epsilon_i} a_i^{-1}, \epsilon_i = \pm 1\right\}.
$$

We claim that for every $k > 0$, there exists $g \in [G_p, G_p]$ such that $\| g \|^* \leq 9 \cdot 3^k$ and $1 \neq g \in \stabilizer_{G_p}(k)$.
We proceed by induction on $k$; for the base case one can select $g = c$.

For the inductive step, we use, as our inductive hypothesis, that there exists $g \in [G_p , G_p]$ such that $\| g \|^* \leq 9\cdot 3^k$ and $1 \neq g \in \stabilizer_{G_p}(k)$.
Then, by assumption,
$$
g = \prod_i a_i [y,x]^{\epsilon_i} a_{i}^{-1},
$$
where $\epsilon_i = \pm 1$.
Let $w_i$ be the words with $w_i(x,y) = a_i$.
Consider the element
$$
u = \prod_i w_i(y, y^x) c^{\epsilon_i} w_i(y,y^x)^{-1},
$$
where $\epsilon_i = \pm 1$.
It is straightforward to see that $c = ([y, x], 1, \ldots, 1)$ (c.f. page 387 of \cite{MR696534}).
Moreover, $y = (x, x^{-1}, 1, \ldots, 1, y)$ and $y^x = (y, x, x^{-1}, 1, \ldots, 1)$.
 Thus, we have
$$
u = ( g, 1, \ldots, 1).
$$
It follows that $u \in \stabilizer_{G_p}(k+1)$.
Further, writing $w_i := w_i(y,y^x)$ and $c = [x,y] y^{-1} [y,x] y$, we have
$$
u = \prod_i w_i ( [x,y]y^{-1} [y,x] y)^{\epsilon_i} w_i^{-1}.
$$
For each $i$, we have either $\epsilon_i = 1$ and
$$
w_i ( [x,y] y^{-1} [y,x] y)^{\epsilon_i} w_i^{-1} 
= 
w_i [x,y] w_i^{-1} w_i  y^{-1} [y,x] y w_i^{-1}
$$
or $\epsilon_i = -1$ and
$$
w_i ( [x,y] y^{-1} [y,x] y)^{\epsilon_i} w_i^{-1} 
= w_i ( y^{-1} [x,y] y [y,x] ) w_i^{-1}
= 
w_i y^{-1} [x,y] y w_i^{-1} w_i [y,x] w_i^{-1}.
$$

Thus, we have
$$
\| u \|_* \leq \sum_i (4\| w_i \| + 10).
$$
Since $y^x$ is part of our generating set, we get
$$
\sum_i (4\| w_i \| + 10) = \sum_i (4 \| a_i \| + 10),
$$
which is clearly less than or equal to $3 \| g \|^* = \sum_i (6\| a_i \| + 12)$.
Thus, $\| u \|^* \leq 3 \| g \|^* \leq 9 \cdot 3^{k+1}$, as desired.

Using the above claim, we get a sequence of points $\{ g_k \}_{k=1}^\infty$ in $G_p$ with $\| g_k \| \leq 9 \cdot 3^k$ and $g_k \in \stabilizer_{G_p}(k)$.
By the proof of Theorem \ref{MainResultLower}, we have
$$
F_{G_p} (n) \succeq 2^{n^{\frac{1}{\log_p(3)}}} = 2^{n^{\frac{\log(p)}{\log(3)}}}.
$$
Selecting $p$ such that $\log(p)/\log(3) > r$ finishes the proof.
\end{proof}

We finish the section with a proof of Proposition \ref{prop:perv}.
We first recall the definition of the group, and show that it is contracting.

Pervova \cite{MR2308183} has constructed the first examples of groups acting on rooted trees which fail to have the congruence subgroup property. Her examples include the following one. Consider a rooted $3$-regular tree $T_3$ and its  automorphisms $a$, $b$ and $c$ defined via recursions 
\begin{align*}
a=\sigma=(1,2,3),\: b=(a,a^{-1},b),\:  c=(c,a,a^{-1}).
\end{align*}
Let $G$ be the group generated by $a$, $b$ and $c$. 

\begin{lem}
The Pervova group $G$ is contracting.
\label{Contract-Pervova}
\end{lem}
\begin{proof}
We use a standard argument to show that the contraction coefficient of $G$ is strictly less than $1$.
 
All words of length two up to taking the inverse are $ab$, $a^{-1}b$, $ab^{-1}$, $a^{-1}b^{-1}$, $ac$, $a^{-1}c$, $ac^{-1}$, $a^{-1}c^{-1}$, $bc$, $bc^{-1}$, $b^{-1}c$, $b^{-1}c^{-1}$.

For these words
\begin{align*}
&a^{\pm 1}b=\sigma^{\pm 1}(a,a^{-1},b)\\
&a^{\pm 1}b^{-1}=\sigma^{\pm 1}(a^{-1},a,b^{-1}) \\
&a^{\pm 1}c=\sigma^{\pm 1}(c,a,a^{-1})\\
&a^{\pm 1}c^{-1}=\sigma^{\pm 1}(c^{-1}, a^{-1}, a)\\
&bc=(ac, 1, ba^{-1})=(\sigma(c,a,a^{-1}), 1, (a,a^{-1},b)\sigma^{-1} )\\
&bc^{-1}=(ac^{-1}, a, ba)=(\sigma(c^{-1},a^{-1},a), a, (a,a^{-1},b)\sigma)\\
&b^{-1}c=(a^{-1}c, a^{-1}, b^{-1}a^{-1})=(\sigma^{-1}(c,a,a^{-1}), a^{-1}, (a^{-1}, a, b^{-1})\sigma^{-1})\\
&b^{-1}c^{-1}=(a^{-1}c^{-1}, 1, b^{-1}a)=(\sigma^{-1}(c^{-1}, a^{-1},a), 1, (a^{-1},a,b^{-1})\sigma)
\end{align*}
and we have reduction of the length on the second level by $\frac{1}{2}$. 

Suppose by induction that for every word $g$ of length $l$ with $1\leq l\leq n-1$ we have $$||g|_v||\leq 1+\frac{||g||}{2}$$ for every vertex $v$ of the second level of the tree. Let $g$ be a word of length $n$ in $G$. Then $g$ can be written as a product $g=g_1 \cdot g_2$ where the length of $g_1$ is $2$ and the length of $g_2$ is $n-2$. For every vertex $v$ of level $2$ in the tree we calculate 
\begin{equation*} ||g|_v||=||(g_1\cdot g_2)|_v||=||g_1|_{g_2(v)} \cdot g_2|_v||\leq ||g_1|_{g_2(v)}||+||g_2|_v||\leq 1+\frac{n}{2}.
\end{equation*}

Suppose $v$ is a vertex of level $k$ with $k\geq 2$, then $v$ can be written as a product $v_0 v_1\dots v_m$ where $m=\lfloor \frac{k}{2}\rfloor$ and $|v_i|=2$ for $1\leq i\leq m$ and $|v_0|<2$. 

For every $g\in G$

\begin{equation*}
||g|_v||=||g|_{v_0 v_1\dots v_m}||<1+\frac{||g|_{v_0 v_1\dots v_m}||}{2}<1+\frac{1}{2}(1+\frac{1}{2}(\dots +(1+\frac{1}{2}||g|_{v_0}||)))< 2+\frac{||g|_{v_0}||}{2^m}\leq 2+\frac{||g||}{2^m}.
\end{equation*}

Thus $G$ is contracting with the contraction coefficient $\lambda\leq \frac{1}{\sqrt{2}}$.
\end{proof}

\begin{proof}[Proof of Proposition \ref{prop:perv}.]
Since the Pervova group contains as a subgroup the Gupta-Sidki 3-group, which is known to be a just-infinite regular branch contracting group with the congruence subgroup property (see references in Example \ref{ManyExamples}), we conclude by applying Theorems \ref{MainResultUpper} that $F_G (n)~\succeq~ 2^{n^{\frac{1}{\log_3(\delta)}}}$.

The group $G$ is regular branch over the subgroup $[G,G]$, see \cite{MR2891709}.  We showed in Lemma~\ref{Contract-Pervova} that the group $G$ is contracting thus $F_G\preceq 2^{n^{\frac{1}{\log_3(1/\lambda)}}}$ by Theorem \ref{MainResultUpper}. We conclude that the Pervova group has near exponential $\Phi_H$ growth as desired.

\end{proof}




\bibliography{refs}
\bibliographystyle{alpha}

\end{document}